\documentclass[reqno,12pt]{amsart}

\usepackage[usenames]{color}
\usepackage{enumerate}
\usepackage{amssymb}
 \usepackage{amsmath,amsthm,amsfonts,amssymb,enumerate,amscd}
\usepackage{float}

\usepackage{pgf,tikz,soul}
\usetikzlibrary{arrows,shapes}
\usetikzlibrary[arrows]
\usepackage{pgfplots}
\usepgfplotslibrary{fillbetween}
\usetikzlibrary{patterns}
\usepackage{pgfplots}
\usepackage{verbatim}

\newtheorem{theorem}{Theorem}[section] % 1st argument is your name for it
\newtheorem{lemma}[theorem]{Lemma}     % 2nd argument is what is printed
\newtheorem{corollary}[theorem]{Corollary}

\newtheorem{example}[theorem]{Example}

\newtheorem{problem}[theorem]{Problem}
\newtheorem{remark}[theorem]{Remark}
\definecolor{verdec}{rgb}{0.76,1.00,0.76}
\definecolor{lila}{rgb}{0.50,0.00,0.50}
\definecolor{rojoc}{rgb}{1.00,0.56,0.56}
\definecolor{lila1}{rgb}{1,0,1}{}
\definecolor{azulc}{rgb}{0.81,1.00,1.00}
\definecolor{lilac}{rgb}{0.78,0.56,0.79}

\def\NN{\mathbb N}

\title[Some examples of $m$-isometries. ]% end with percent
 {Some examples of $m$-isometries}% This is the full title of the paper

\author[Teresa Berm\'udez]{Teresa Berm\'udez}
\email{tbermude@ull.es }%\hfill Examples of-m-isometries-3.tex}
\author{Antonio Martin\'on}
\email{anmarce@ull.es}
\author{Hajer Zaway}
\email{hajer$_{-}$zaway@live.fr}

\address{
  Departamento de An\'{a}lisis Matem\'{a}tico\\
  Universidad de La Laguna\\
  38271 La Laguna (Tenerife), Spain}

\address{Department of Mathematics\\ Faculty of Sciences, University of Gab\`{e}s\\
6072 Gab\`{e}s, Tunisia}

\date{\today }

\begin{document}

\maketitle

\begin{abstract}
We obtain  the admissible sets on the unit circle to be the spectrum of  a strict $m$-isometry on an $n$-finite dimensional Hilbert space. This property gives a better picture of the correct spectrum of an $m$-isometry.
We  determine that the only $m$-isometries on $\mathbb{R}^2$ are $3$-isometries  and isometries giving by $\pm I+Q$, where $Q$ is a nilpotent operator. Moreover, on  real Hilbert space, we obtain that $m$-isometries preserve  volumes.
Also we present a way to construct a strict $(m+1)$-isometry with an $m$-isometry given, using ideas of Aleman and Suciu \cite[Proposition 5.2]{AS} on infinite dimensional Hilbert space.
\end{abstract}

\section{Introduction}

Let $H$ be a Hilbert space. Denote by $L(H)$ the algebra of  bounded linear operators on $H$. For $T \in L(H)$ we consider the adjoint operator $T^* \in L(H)$,  which is the unique map that satisfies
$$
\langle Tx , Ty \rangle = \langle T^*Tx , y \rangle  \; ,
$$
for every $x,y \in H$.
 Given $T\in L(H)$, denote by $Ker (T)$ and $R(T)$, the kernel and range of $T$, respectively.
For a  positive integer $m$, an {\it $m$-isometry} is an operator $T \in L(H)$ which satisfies the condition
\begin{equation}\label{m-isom H}
(yx-1)^m(T):=\sum_{k=0}^m (-1)^{m-k} {m\choose k}  T^{*k} T^k = 0 \; ;
\end{equation}
equivalently
\begin{equation}\label{m-isom H x}
\sum_{k=0}^m (-1)^{m-k} {m\choose k}  \| T^{k} x \|^2 = 0 \; ,
\end{equation}
for every $x\in H$. A {\it strict $m$-isometry} is an $m$-isometry which is not an $(m-1)$-isometry.
This class of operators was introduced by Agler in \cite{ad} and was studied by Agler and Stankus in
 \cite{as1,as2,as3}.

Let $n$ be a positive integer. Recall that $Q \in L(H)$ is {\it $n$-nilpotent} if $Q^n=0$ and $Q^{n-1} \neq 0$.

A notion related with  $m$-isometries is the following. An operator $T \in L(H)$ is {\it isometric $n$-Jordan} if there exist an isometry $A \in L(H)$ and an $n$-nilpotent $Q \in L(H)$ such that $T = A + Q$ with  $AQ=QA$.

\begin{theorem}\cite[Theorem 2.2]{bmn}\label{general}
Any isometric $n$-Jordan operator is a strict $(2n-1)$-isometry.
\end{theorem}

Actually, a much stronger result is true. Indeed
in  \cite [Theorem 3]{bmno 3}, it is obtained a generalization of Theorem \ref{general} for $m$-isometries:  if $T$ is an $m$-isometry, $Q$ is an $n$-nilpotent operator and they commute, then $T+Q$ is a $(2n+m-2)$-isometry. See also \cite{gs, l}. Moreover, the study of isometric $n$-Jordan operators concerning  with $m$-isometries  on Banach spaces context has been studied in \cite{bmno 3}.

Another way of generalization was obtained in \cite[Proposition 2.6]{bmn} for  sub-isometry $n$-Jordan operator. Recall that $T$ is a \emph{sub-isometry $n$-Jordan} operator  if $T$ is  the restriction of an isometry $n$-Jordan operator $J$ to an invariant subspace of $J$.

Notice that  Theorem \ref{general} gives  an  easy way to construct examples of $m$-isometries,  for an  odd $m$. It is sufficient to choose  the identity operator as the isometry and any $n$-nilpotent operator with $n=\frac{m+1}{2}$.

At a first  glance, we could think that  all the $m$-isometries  come from isometric $n$-Jordan. However, this is not true, since there are strict $m$-isometries for even $m$, see \cite[Proposition 9]{at}. What can we say about $m$-isometries with odd $m$? Recently, Yarmahmoodi and  Hedayatian have proven that the only isometric $n$-Jordan weighted shift operators are isometries \cite[Theorem 1]{yh}. So, there are $m$-isometries that are not isometric $n$-Jordan, since Athavale  in \cite{at} gave examples of strict $m$-isometries with the weighted shift operator for all integers $m$.

Whenever, if  $H$ is finite dimensional is possible to say   more.

Some authors have given examples of $m$-isometries. For example with the unilateral or bilateral  weighted shift \cite{al, bmne, bmno, m} and with the composition operator \cite{bmno, b, gc}. Another way to construct examples of $m$-isometries is developing different tools like tensor product \cite{d}, functional calculus \cite{gu}, on Hilbert-Schmidt class \cite{bjz} and with $C_0$-semigroups \cite{bah, eva, r}.

\ \par

The purpose of this paper is to make a clear picture of $m$-isometries on finite dimensional Hilbert space. In Section 2,  we begin with the study of $m$-isometries on $\mathbb{R}^2$ and on $\mathbb{R}^n$, with $n\geq 3$. We give all the 3-isometries on $\mathbb{R}^2$. Also, we obtain  the expression of $m$-isometries and  study how this class of operators change  volumes on $\mathbb{R}^n$. Moreover, we study the case of complex Hilbert space, where we prove the admissible sets on the unit circle to be the spectrum of an $m$-isometry.
In Section 3, we reproduce  similar  ideas of Aleman and Suciu \cite[Proposition 5.2]{AS} to define a 3-isometry using a given 2-isometry. In fact, we obtain a way to construct a strict $(m+1)$-isometry using a weaker condition than a strict $m$-isometry.

In particular, we will  answer the following  problems.

\begin{problem}\label{trivial} Let $T\in L(H)$  with $H$ an $n$-finite dimensional Hilbert space and $m$ an odd integer. Are all  strict  $m$-isometries  of the form  $ \lambda I + Q$, where $Q$ is a nilpotent operator and $\lambda $ is a complex number with modulus 1?
\end{problem}

\begin{problem}\label{area}
Let $T\in L(\mathbb{R}^n)$. How does an $m$-isometry $T$  change volumes?
\end{problem}

\begin{problem}\label{le}
Let $H$ be any $n$-finite dimensional Hilbert space and let $T$ be an $m$-isometry with odd $m$.
 What can we say about the spectrum?
\end{problem}

%%%%%%%%%%%%%%%%%%%%%%%%%%%%%%%%%%%
%%%%%%%%%%%%%%%%%%%%%%%%%%%%%%%%%%%%%%%%%%%%
%%%%%%%%%%%%%%%%%%%%%%%%%%%%%%%%%%%%%%%%%%%%%%%%%%%%%

\section{$m$-isometries on  finite dimensional Hilbert space}

Recall some important properties of the spectrum of an  $m$-isometry.

Denote $\overline {\mathbb{D}}$ and $\partial \mathbb{D}$ the closed unit disk and the unit circle, respectively.
\begin{lemma}\label{spectrum of m-isometry}
Let $m$ be a positive integer, $H$ be a Hilbert space and $T \in L(H)$ be  an $m$-isometry. Then
\begin{enumerate}
\item[$(1)$] { \rm \cite[Lemma $1.21$]{as1}} $\sigma(T)=\overline {\mathbb{D}}$ or $\sigma (T) \subseteq \partial \mathbb{D}$.
\item[$(2)$] {\rm \cite[Lemma $19$]{ahs}} The eigenvectors of $T$ corresponding to distinct eigenvalues are orthogonal.
\end{enumerate}
\end{lemma}

\begin{remark}{\rm  \begin{enumerate}
\item[(1)] Notice that any $m$-isometry on a finite dimensional space is bijective.
\item[(2)] It is well known that if $Q$ is $k$-nilpotent on an $n$-dimensional vector space, then $k \leq n$.
\end{enumerate}
}
\end{remark}

Denote
$$
I_m(H):=\{ T\in L(H) \;\; : \;\;  T \mbox{ is an $m$-isometry} \} \;.
$$

The following theorem gives a nice picture of  $m$-isometries on finite dimensional spaces.

\begin{theorem}\label{Jordan} (\cite[Theorem 2.7]{bmn},$\;$\cite[page 134]{ahs}) Let $H$ be an $n$-finite dimensional Hilbert space and $T \in L(H)$. Then
\begin{enumerate}
\item $T$ is a strict $m$-isometry if and only if $T$ is an isometric $k$-Jordan operator, where  $m =2k-1$ with $k\leq n$.

\item $I_1(H)=I_2(H)\subsetneq I_3(H)=I_4(H) \subsetneq \ldots \subsetneq I_{2n-1}(H)=I_j(H)$ for all $j\geq 2n-1$.
\end{enumerate}
\end{theorem}
\begin{proof}
We include the proofs for completeness.

{\it (1)} Assume that $T$ is a strict  $m$-isometry on $H$. Then the spectrum of $T$, $\sigma (T)=\{ \lambda_1, \lambda_2, \ldots, \lambda_s\}$, where  $\lambda _i $ are eigenvalues of modulus 1, since  the spectrum of $T$ must be in the unit circle and $m$ is odd \cite[Lemma 1.21 \& Proposition 1.23]{as1}. By part (2) of Lemma \ref{spectrum of m-isometry}, the spectral subspaces of $T$, $H_i:=Ker (T-\lambda _i)^{n_i}$ are mutually orthogonal and
$$
T\cong T_{|H_1} \oplus  \cdots \oplus T_{|H_s}\;,
$$
where  $n_1, \ldots , n_s$ are positive integers such that  $Ker(T-\lambda _i)^{n_i}=Ker(T-\lambda _i)^N$ for all $N\geq n_i$. Moreover, for all $j\in \{1, \ldots , s\}$, we have that $\sigma (T_{|H_j})=\{ \lambda _j\}$ and $T_{|H_j}$ is of  the form $\lambda _j + Q_j$ for some nilpotent operator $Q_j$. So, $T=A+Q$ for some isometry, in fact unitary diagonal operator $A$ and  some nilpotent operator $Q$ such that  $AQ=QA$.

The converse is consequence of Theorem \ref{general}.

{\it (2)} Let us prove that $I_{2\ell -1}(H)=I_{2\ell}(H)$ for all $\ell \in \NN$. Recall that if $T$ is $(2\ell)$-isometry, then $T$  is bijective and so $T$ is $(2\ell -1)$-isometry \cite[Proposstion 1.23]{as1}. Moreover, the highest degree of nilpotent operator on $n$-dimensional Hilbert space is $n$. The result is a consequence of  Theorem \ref{general}.
\end{proof}

\subsection{$m$-isometries on real Hilbert spaces}

Next, we study the $m$-isometries on $\mathbb{R}^n$.

Based on  the above results, we obtain all   $m$-isometries on $\mathbb{R}^2$.

\begin{theorem}\label{isom R2} If $T \in L(\mathbb{R}^2)$ is a strict $m$-isometry, then $m=1$ or $m=3$ and $T= A+Q$, where $A$ is an isometry and $Q$ is a nilpotent operator of order 2 that commutes.
\end{theorem}

Recall that isometries  on $\mathbb{R}^2$ are given by
$$
R_\theta := \left(
\begin{array}{cc}
\cos \theta  &   -\sin \theta \\
\sin \theta  &   \cos \theta
\end{array}
\right)
\;
 \mbox{ and }
S_\theta := \left(
\begin{array}{cc}
\cos \theta       &      \sin \theta \\
\sin \theta       &      -\cos \theta
\end{array}
\right)
\; ,
$$
where
\begin{enumerate}
\item $R_\theta$ is a {\it rotation} (about $0$) and its determinant, $\det (R_\theta)$ is $1$ and
\item $S_\theta$ is a symmetry respect to the straight  line of equation $x_2 = \tan (\theta/2) x_1$ and $\det (S_\theta)=-1$.
\end{enumerate}

And the non-zero  nilpotent operators on $\mathbb{R}^2$ are $\lambda M, \; \lambda N$ and $\lambda Q_k$ where
\begin{equation}\label{nilpotentes-2}
M := \left(
\begin{array}{cc}
0  & 1 \\
0 &  0
\end{array}
\right)
\;, \;\;\;\;\;
N := \left(
\begin{array}{cc}
0  &  0 \\
1 &   0
\end{array}
\right)
\;
Q_k := \left(
\begin{array}{cc}
1  &  k \\
-\frac{1}{k} &  -1
\end{array}
\right)
\; ,
\end{equation}
with  $k \neq 0$ and $\lambda \in \mathbb{C}\setminus\{ 0\}$.

We are interested in studying  isometries that commute with nilpotent operators on $\mathbb{R}^2$.

\begin{lemma}\label{lema-conmuta}
The unique isometries on $L(\mathbb{R}^2)$ that commute with a non-zero  nilpotent operator are the trivial cases, that is, $\pm I$.
\end{lemma}
\begin{proof}
Simple calculations prove that
$$
R_\theta M = M R_\theta \Longleftrightarrow R_\theta N = N R_\theta \Longleftrightarrow R_\theta Q_k = Q_k R_\theta \Longleftrightarrow \sin \theta = 0 \Longleftrightarrow \theta = 0 \mbox{ or } \theta = \pi \; .
$$
That is, the unique isometries of type $R_\theta$ which commute with some non-zero nilpotent (hence with all the nilpotent) are $R_0 = I$ and  $R_\pi = -I$.

Analogously, we have that
$$
S_\theta M = M S_\theta \Longleftrightarrow S_\theta N = N S_\theta \Longleftrightarrow S_\theta Q_k = Q_k S_\theta \Longleftrightarrow \sin \theta = \cos \theta = 0 \; ,
$$
which  it is impossible. Hence there are not isometries $S_\theta$ which commute with some non-zero nilpotent operator.
\end{proof}

Taking into account Theorem \ref{Jordan} we give  the   unique strict 3-isometries on $\mathbb{R}^2$. Indeed, we  answer  Problem \ref{trivial}  for $n=2$ in the following result.

\begin{theorem}\label{caso-2}  The strict $3$-isometries  on $\mathbb{R}^2$ are of the form $\pm I+Q$, where $Q$ is a non-zero nilpotent operator given in (\ref{nilpotentes-2}).
\end{theorem}

\begin{proof}
 It is immediate by Theorem \ref{isom R2} and Lemma  \ref{lema-conmuta}.

\end{proof}

Let $T\in L(\mathbb{R}^n)$ with $n\geq 3$ and let us consider the following $n$ conditions:

\begin{enumerate}
\item[$(M_k)$] $\;\;\;\;\;\;\;\;\;\;\;\;\;\;\;\;S_k(Tx_1, Tx_2, \ldots , Tx_k)=S_k(x_1, x_2, \ldots , x_k)$
\end{enumerate}
\noindent for all $x_1, x_2, \ldots, x_k \in \mathbb{R}^n$ and $k=1, 2, \ldots , n$, where
$S_k(x_1, x_2, \ldots , x_k) $ denotes the $k$-dimensional measure of the set
$$
\{ \lambda _1 x_1 + \lambda _2x_2+ \ldots + \lambda _k x_k\;\;\; :\;\;\; 0\leq \lambda _i \leq 1,\;\; \mbox{ for } i=1, 2, \ldots , k\} \;.
$$

\begin{lemma}\label{condicion}
Let $T\in L(\mathbb{R}^n)$. Then
\begin{enumerate}
\item \cite[Teorema II]{G} $T$ satisfies the  conditions $(M_1), \; (M_2),\; \cdots \; , (M_{n-1})$ if and only if  $T$ is an isometry.
\item \cite{F} The condition $(M_n)$ is equivalent to $det (T)=\pm 1$.
\end{enumerate}
\end{lemma}

An easy application of Theorem \ref{general} gives that, for example in $\mathbb{R}^3$, we have strict 3-isometries giving by $\pm I +Q$, where $Q$ is a 2-nilpotent operator and strict 5-isometries giving by $\pm I +Q$, where $Q$ is a 3-nilpotent operator.

The next result gives answer to  Problems \ref{trivial} and \ref{area} for $n\geq 3$, where $n$ is the dimension of the Hilbert space.

\begin{theorem}\label{caso-n}
Let  $n\geq 3$. Then the following properties follow:
\begin{enumerate}
\item There are non-trivial  strict $m$-isometries on $L(\mathbb{R}^n)$ for  any odd $m$ less than  $2n-1$, that is, there exists an isometry different from $\pm I$ such that commutes with a non-zero $k$-nilpotent operator with $k\in \{ 1, 2, \cdots, n-1\}$.

\item The $m$-isometries preserve volumes.
\end{enumerate}
\end{theorem}
\begin{proof}
{\it (1)} Define
\begin{eqnarray*}
A(x_1,x_2, \ldots, x_n):&=& (-x_1, x_2, \ldots , x_n)\\
Q_j(x_1,x_2, \ldots, x_n):&=& (0, x_3, x_4, \cdots, x_{j+1}, 0, \cdots, 0) \;.
\end{eqnarray*}
Then $A$ is an isometry and $Q_j$ is a $j$-nilpotent operator such that
$$
AQ_j(x_1, x_2, \ldots, x_n)=Q_jA(x_1, x_2, \ldots, x_n)=(0, x_3, x_4, \ldots , x_{j+1}, 0, \ldots, 0) \;
,
$$
for all $(x_1, x_2, \ldots , x_n)\in \mathbb{R}^n$. By Theorem \ref{general}, we get that $A+Q_j$ is a   non trivial strict  $(2j-1)$-isometry for $j=1, \ldots, n-1$.

{\it (2)} By Lemma \ref{condicion}, it will be enough to prove that $det(A+Q)=\pm 1$ for all isometries $A$ that commute with a nilpotent operator $Q$. Since $AQ=QA$, then $\sigma (A+Q)=\sigma(A)$ by \cite[Proposition 1.1]{yhy}. According to the spectrum of an isometry on a finite dimensional space, we have that the spectrum of $A$ is a closed subset of the unit circle. By \cite[page 150]{a}, the determinant of $T$ is the product of the eigenvalues of $T$, counting multiplicity. Hence $det(T)=\pm 1$.

\end{proof}

The converse of part (2) of Theorem \ref{caso-n} is not true, as prove the following example.

\begin{example}
Let $T:=\left( \begin{array}{ccc}
1&0&0\\
0&2& 1\\
0&1& 1
\end{array}\right)$. Then $det(T)=1$ and $T$ is not a 3-isometry, since
$$
\|T^3x\|^2-3\|T^2x\|^2+3\|Tx\|^2-\|x\|^2\neq 0\;,
$$
for  $x: =(1,1,0)$.
\end{example}

\subsection{On complex Hilbert space}

We recall the following results about the spectrum of $m$-isometries.

\begin{lemma}{ \rm \cite[Theorem 4.4]{bmn}}\label{spectrum}
Let $H$ be an  infinite dimensional Hilbert space.
\begin{enumerate}
\item If $K$ is any compact subset of $\partial \mathbb{D}$, then there exists  a strict $m$-isometry for any odd number $m$ such that $\sigma (T)=K$.
\item If $K$ is the closed unit disk, then there exits a strict $m$-isometry for any integer number $m$.
\end{enumerate}
\end{lemma}

The main aim of this section is to solve Problem \ref{le}.

Let $T\in L(\mathbb{C}^n)$ be an $m$-isometry.  It is clear that $\sigma (T)\subseteq \partial \mathbb{D} $ by part (1) of Lemma \ref{spectrum of m-isometry} and $\sigma (T)$ has at most $n$ different eigenvalues. Indeed if $K:=\{ \lambda _1, \cdots, \lambda _ n\}$ with $\lambda _i$ different complex numbers on the unit circle, then it is possible to define  an isometry $T$  such that $\sigma (T)=K$. In particular, the following operator
$$
T(x_1, \cdots,  x_n):=(\lambda_1x_1, \cdots, \lambda_nx_n)
$$
is an isometry on $\mathbb{C}^n$ with $\sigma (T)=\{ \lambda_1, \cdots, \lambda_n\}$.

In the following theorem we prove that any $m$-isometry with $m\geq 3$ on $\mathbb{C}^n$ can not have $n$ different eigenvalues.

\begin{theorem}
Any strict $(2k-1)$-isometry on $\mathbb{C}^n$ with $2 \leq k \leq n$ has at most $n-1$ distinct eigenvalues.
\end{theorem}

\begin{proof}
Assume that $T\in L(\mathbb{C}^n)$ is a strict $(2k-1)$-isometry with $\sigma(T) = \{\lambda_1, ..., \lambda_n\}$
where $\lambda_1, ..., \lambda_n$ are  different eigenvalues of $T$. Then  $T$ could be written as
$T = PSP^{-1}$, for some $P\in L(\mathbb{C}^n)$ where
$$
S:=\left(\begin{array}{cccc}
                                                     \lambda_1 &     0     &   \dots       & 0\\
                                                          0    & \lambda_2 &    \ddots     & \vdots\\
                                                     \vdots    &  \ddots   &    \ddots     &     0\\
                                                         0     &  \dots    &      0        &   \lambda_n
                                                        \end{array}\right)
$$
and $|\lambda _i|=1$ for $i\in \{ 1, \cdots, n\}$, by  part (1) of Lemma \ref{spectrum of m-isometry}.
 Since $T$ is a strict $(2k-1)$-isometry, by part (2) of Lemma \ref{spectrum of m-isometry}, the operator  $P$ is a unitary operator. This means that $T$ is unitarily equivalent to $S$, therefore $T$ is a unitary, which is a contradiction.
\end{proof}

\begin{theorem}\label{characteristic of m-isometry}
The strict $(2k-1)$-isometries on $\mathbb{C}^n$, with $2 \leq k \leq n$ are of the form $(\lambda_1 I_{n_1} \oplus ...\oplus \lambda_\ell I_{n_{\ell}}) + Q$, with  $\ell \in \{1,...,n-k+1\}$, where $Q$ is a  $k$-nilpotent,  $|\lambda_j|= 1$ for all $j \in \{1,...,\ell\}$ and $n_{1}+...+n_{\ell}=n$.
\end{theorem}

\begin{proof}
Suppose that $T$ is a strict $(2k-1)$-isometry. By Theorem \ref{Jordan}, we have that $T = U +Q$, where  $U$ is a unitary  operator and $Q$ is a $k$-nilpotent operator such that $UQ = QU$.\\
Assume, by contradiction, that $T$ has at least $n-k+2$ distinct eigenvalues. That means
                     $$\sigma(T) = \{\lambda_1, ..., \lambda_r\}, \;\ \mbox{with} \;\ r\geq n-k+2.$$
Then $\mathbb{C}^n = H_{\lambda_1}\oplus ...\oplus H_{\lambda_r}$ , where $H_{\lambda_i}:= Ker(T - \lambda_i I)^{n_i}$ and ${n_i}$ is the order of multiplicity of the eigenvalue $\lambda_i$. Denote $T_{|H_i}$ the restriction operator  of $T$ to $H_i$,  for $1\leq i\leq r$. Then
 $T_{|H_i}=\lambda_i I_{n_i} +Q_i$, where $Q_i$ is a $h_i$-nilpotent with $1 \leq h_i \leq n_i$. By part (2) of Lemma \ref{spectrum of m-isometry},  we conclude that $T$ could be written as
 $$
 T= (\lambda_1 I_{n_1} \oplus ...\oplus \lambda_rI_{n_r}) + (Q_1 \oplus ...\oplus Q_r)\; ,
 $$
where $Q_1 \oplus ...\oplus Q_r$ is a $k_0$-nilpotent, with $k_0:=\max_{i=1, ..., r} \{h_i\}$ and $k_0 < k$.  Then we get a contradiction.
\end{proof}

\begin{corollary}
If $T \in L(\mathbb{C}^n)$ is a strict $(2k-1)$-isometry, with $2 \leq k \leq n$, then $\sigma(T) \subseteq \{\lambda_1, ..., \lambda_{n-k+1}\} \subseteq \partial \mathbb{D}$.
\end{corollary}

\begin{corollary}
Any $(2n-1)$-isometry on $\mathbb{C}^{n}$ is of the form $\lambda I+Q$, where $Q$ is an $n$-nilpotent operator and $\lambda \in\partial \mathbb{D}$. In particular the spectrum is   a single point on the unit circle.
 \end{corollary}

%%%%%%%%%%%%%%%%%%%%%%%%%%%%%%%%%%%
%%%%%%%%%%%%%%%%%%%%%%%%%%%%%%%%%%%%%%%%%%%%
%%%%%%%%%%%%%%%%%%%%%%%%%%%%%%%%%%%%%%%%%%%%%%%%%%%%%

\section{Construction  of an $(m+1)$-isometry from  an $m$-isometry }

In this section we present a method to construct a Hilbert space $H_k$ and an $(m+1)$-isometry  on $H_k$ from an $m$-isometry $T^k$ on a Hilbert space for some integer $k$.  Our result is based  on the construction  given by Aleman and Suciu in \cite[Proposition 5.2]{AS}  for $m=2$ and $k=1$.

Henceforth $H$ will denote an infinite dimensional Hilbert space.

Given $S\in L(H)$, $x\in H$ and an integer $\ell \geq 1$, it is defined
$$
\beta_{\ell}(S,x):=\frac{1}{\ell !}\sum_{j=0}^\ell  (-1)^{\ell -j} {\ell  \choose j} \| S^j x\|^2 \;.
$$
Note that $S$ is an $m$-isometry  if and only if $\beta _m(S,x)=0$ for all vector $x\in H$.

Consider $\mathbb{C}[z]$ the space  of all complex polynomials. Given $p\in \mathbb{C}[z]$,  we write
$$
p(z)=\sum_{n\geq 0} p_nz^n
$$
and define $Lp\in \mathbb{C}[z]$ in the following  way:
$$
Lp(z):= \sum_{n\geq 1}p_n z^{n-1}=\frac{p(z)-p_0}{z}\;.
$$
We have that $\mathbb{C}[z]$ is an inner product  space with the norm $\|. \|_2$ given by
$$
\|p\|^2_2:=\sum_{n\geq0}|p_n|^2 \;.
$$
Also if we consider a new norm on $\mathbb{C}[z]$ defined by
$$
\||p\||^2_k:=\|p\|^2_2+\sum_{n\geq0}\|(L^{nk}p)(T)x_0\|^2\;,
$$
it is obtained that $\mathbb{C}[z]$ is an inner product space  with $\|| . \||_k$. Denote   $H_k$ its completion with the new norm.

The following combinatorial result will be useful.
\begin{lemma}{ \rm \cite[Eq. $0.151$ $(4)$]{GR}}\label{binom 2}
If $m$ is any positive integer, then
$$
\sum_{k=0}^{m}(-1)^{k} \binom{n}{k}=(-1)^m \binom{n-1}{m}\;,
$$
for any integer $n\geq m+1$.
\end{lemma}

Recall that the class of $m$-isometries is  stable under  powers.   However, the converse is not true. See \cite{bdm,j}.

\begin{theorem}
Let $T\in L(H)$ such that $T^k$ is a strict $m$-isometry on $R(T^k)$, for some $k$, and $x_0 \in H \setminus  \{0\}$ such that $\beta_{m-1} (T^{k},T^k x_0) \neq 0$.
\begin{enumerate}
\item For every $p\in \mathbb{C}[z]$ and $j\in \mathbb{N}$,
$$
\||M^{kj}_zp\||^2_k=\||p\||^2_k+\sum_{i=1}^j\|T^{ki}p(T)x_0\|^2\; ,
$$
where $M_z$ denotes the multiplication operator defined by $M_zp:=zp$.
\item For every $p\in \mathbb{C}[z]$ and $\ell\geq 1$,
\begin{equation}\label{beta}
\beta _{\ell+1}(M^k_z,p)= \frac{\ell !}{(\ell +1)!} \beta _{\ell } (T^k, T^kp(T)x_0)\;.
\end{equation}
\item The extension of $M_z^k$  to  $H_k$ is an $(m+1)$-isometry. 
\end{enumerate}
\end{theorem}

\begin{proof}
(1)  Let $p$ be  any polynomial and $j\in \NN$. Then will prove that
 \begin{equation}\label{ll}
  \||M^{kj}_zp\||^2_k=\||p\||^2_k+\sum_{i=1}^j\|T^{ki}p(T)x_0\|^2\;,
 \end{equation}
by induction. For $j=1$ we need to prove that
\begin{equation}\label{ecmz}
\||M^k_zp\||^2_k=\||p\||^2_k+\|T^kp(T)x_0\|^2\;,
\end{equation}
for any polynomial $p$.\\
Let $p(z):=\sum_{n\geq 0}p_n z^n$. Then
\begin{eqnarray*}
\|| M^k_z p\||^2_k = \|| z^kp\||^2_k & = & \| z^kp\|^2 _2 + \sum_{n\geq0}\|(L^{nk}z^kp)(T)x_0\|^2 \\[0.7pc]
                                 & = &
                                 \|p\|^2 _2 + \| (z^kp)(T)x_0\|^2 + \sum_{n\geq1}\|(L^{nk}z^kp)(T)x_0\|^2 \\[0.7pc]
                                 & = & \|p\|^2 _2 + \| T^kp(T)x_0\|^2 + \sum_{n\geq0}\|(L^{nk}p)(T)x_0\|^2 %\\[1pc]
                                 %& = &
                                 =\||p\||^2 _k + \| T^kp(T)x_0\|^2\;.
\end{eqnarray*}
Then (\ref{ecmz}) holds.

Suppose that (\ref{ll}) is true for $j$.
 Let us prove it for $j+1$. Then
 \begin{eqnarray*}
\|| M^{k(j+1)}_z p\||^2_k & =  & \|| M^{kj}_z(M^k_z p)\||^2_k =  \|| M^k_z p\||^2_k + \sum_{i=1}^j\|T^{ki}(M^k_zp)(T)x_0\|^2 \\[0.7pc]
                        & = &
                        \|| z^kp\||^2_k + \sum_{i=1}^j\|T^{ki}T^kp(T)x_0\|^2 \\[0.7pc]
                        & = & \| p\|^2_2  + \sum_{n\geq0}\|(L^{nk}z^kp)(T)x_0\|^2 + \sum_{i=1}^j\|T^{k(i+1)}p(T)x_0\|^2 \\[0.7pc]
                        & = & \| p\|^2_2  + \| T^kp(T)x_0\|^2 + \sum_{n\geq0}\|(L^{nk}p)(T)x_0\|^2 + \sum_{i=2}^{j+1}\|T^{ki}p(T)x_0\|^2 \\[0.7pc]
                        & = & \|| p\||^2_k  + \sum_{i=1}^{j+1}\|T^{ki}p(T)x_0\|^2\;.
\end{eqnarray*}
So we prove (\ref{ll}).

(2) For $\ell \in \mathbb{N}$, we have
 \begin{eqnarray*}
\beta_{\ell+1}(M^k_z,p) &=& \frac{1}{(\ell+1)!}\sum_{j=0}^{\ell+1}(-1)^{\ell+1-j} \binom{\ell+1}{j}\|| M^{kj}_zp\||^2_k\\
                        &=&
                       \frac{1}{(\ell+1)!}\left((-1)^{\ell+1}\||p\||^2_k+\sum_{j=1}^{\ell+1}(-1)^{\ell+1-j} \binom{\ell+1}{j}\||M^{kj}_zp\||^2_k\right)\\
                        &=&\frac{1}{(\ell+1)!}\left((-1)^{\ell+1}\||p\||^2_k+\sum_{j=1}^{\ell+1}(-1)^{\ell+1-j} \binom{\ell+1}{j}\left(\||p\||^2_k+\sum_{i=1}^{j}\|T^{ki}p(T)x_0\|^2 \right) \right)\\
                        &=& \frac{1}{(\ell+1)!} \sum_{j=1}^{\ell+1}(-1)^{\ell+1-j} \binom{\ell+1}{j}\sum_{i=1}^{j}\|T^{ki}p(T)x_0\|^2\\
                        &=&\frac{1}{(\ell+1)!} \sum_{j=1}^{\ell+1} \|T^{kj}p(T)x_0\|^2 \sum_{i=j}^{\ell +1} (-1)^{\ell+1-i}\binom{\ell+1}{i}\;,
\end{eqnarray*}
where $p$ is any  polynomial.\\
Using Lemma \ref{binom 2}, in the last sum, we have that
$$
\sum_{i=j}^{\ell+1} (-1)^{\ell+1-i} {\ell +1 \choose i}= -\sum_{i=0}^{j-1} (-1)^{\ell +1-j} {l+1 \choose j}=
(-1)^{\ell+j-1} {\ell \choose j-1}\;.
$$
So,
\begin{eqnarray*}
\beta_{\ell+1}(M^k_z,p) &=& \frac{1}{(\ell+1)!} \sum_{j=1}^{\ell+1} \|T^{kj}p(T)x_0\|^2 (-1)^{\ell+j-1}{\ell \choose j-1}\\
                        &=& \frac{1}{(\ell+1)!} \sum_{j=0}^{\ell} (-1)^{\ell-j}{\ell \choose j} \|T^{kj}p(T)T^kx_0\|^2\\
                        &=& \frac{\ell!}{(\ell+1)!}\beta_\ell (T^{k},T^kp(T)x_0)\;.
\end{eqnarray*}
So, (\ref{beta}) is proved.

(3) It is enough to prove that
$\beta _{m+1}(M_z^k, p)=0$ for any $p\in \mathbb{C}[z]$. This is a consequence of (\ref{beta}), since $T^k$ is an $m$-isometry on $R(T^k)$.
\end{proof}

\begin{corollary}\cite[Proposition 5.2]{AS}
Let $T$ be a $2$-isometry on a Hilbert space $H$. Fix $x_0\in H\setminus \{ 0\}$ and let $H_1$ be the completion of the space of analytic polynomials with respect to the norm
$$
\| p\|_1^2:= \|p\|_2^2 + \sum_{n\geq 0} \| (L^np)(T)x_0\|^2 \;.
$$
Then the multiplication operator by the independent variable $M_zp:= zp$ extends to a $3$-isometry  on $H_1$.
\end{corollary}

%%%%%%%%%%%%%%%%%%%%%%%%%%%%%%%%%%%%%%%%%%%%
%%%%%%%%%%%%%%%%%%%%%%%%%%%%%%%%%%%%%%%%%%%%%%%%%%%%%%%%%
%%%%%%%%%%%%%%%%%%%%%%%%%%%%%%%%%%%%%%%%%%%%%%%%%%%%%%%%%%%%%%%%%

{\bf Acknowledgements:}
The first  author is partially supported by grant of Ministerio  de Ciencia e Innovaci\'{o}n, Spain, project no. MTM2016-75963-P. The third author was supported in part by Departamento de Análisis Matemático of  Universidad de La Laguna and Le Laboratoire de Recherche Mathématiques et Applications LR17ES11.

%%%%%%%%%%%%%%%%%%%%%%%%%%%%%%%%%%%%%%%%%%%%%%%%%%%%%%%%%%%%%%%%%%%%%%%%%%%%%%%%%%%%%%%%%%%%
%%%%%%%%%%%%%%%%%%%%%%%%%%%%%%%%%%%%%%%%%%%%%%%%%%%%%%%%%%%%%%%%%%%%%%%%%%%%%%%%%%%%%%%%%%%%
%%%%%%%%%%%%%%%%%%%%%%%%%%%%%%%%%%%%%%%%%%%%%%%%%%%%%%%%%%%%%%%%%%%%%%%%%%%%%%%%%%%%%%%%%%%%


\begin{thebibliography}{99}
%%%%%%%%%%%%%%%%%%%%%%%%%%%%%%%%%%%%%%%%%%%%%%%%%%%%%%%%%%%%%%%%%%%%%%%%%%%%%%%%%%%%%%%%%%%%
%%%%%%%%%%%%%%%%%%%%%%%%%%%%%%%%%%%%%%%%%%%%%%%%%%%%%%%%%%%%%%%%%%%%%%%%%%%%%%%%%%%%%%%%%%%%
%%%%%%%%%%%%%%%%%%%%%%%%%%%%%%%%%%%%%%%%%%%%%%%%%%%%%%%%%%%%%%%%%%%%%%%%%%%%%%%%%%%%%%%%%%%%

\bibitem{al} B. Abdullah, T.  Le,  The structure of $m$-isometric weighted shift operators, {\it Oper. Matrices}, {\bf  10} (2016), no. 2, 319-334.

\bibitem{ad} J. Agler, A disconjugacy theorem for Toeplitz operators, {\it Amer. J. Math.}, \textbf{112} (1990) 1-14.

\bibitem{ahs} J. Agler, W. Helton, M. Stankus. Classification of Hereditary Matrices, {\it Linear Algebra App.}, \textbf{274} (1998) 125-160.

\bibitem{as1} J. Agler, M. Stankus, $m$-isometric transformations of Hilbert space. I, {\it Integral Equations Operator Theory,} \textbf{21} (1995), no. 4, 383-429.

\bibitem{as2} J. Agler, M. Stankus, $m$-isometric transformations of Hilbert space. II, {\it Integral Equations Operator Theory,} \textbf{23}  (1995), no. 1, 1-48.

\bibitem{as3} J. Agler, M. Stankus, $m$-isometric transformations of Hilbert space. III, {\it Integral Equations Operator Theory,} \textbf{24}  (1996), no. 4, 379-421.

\bibitem{AS} A. Aleman, L. Suciu,  On ergodic operator means in Banach spaces, {\it Integral Equations Operator Theory}, {\bf 85} (2016), 259--287.

\bibitem{at} A. Athavale, Some operator theoretic calculus for positive definite kernels, {\it  Proc. Amer. Math. Soc.}, \textbf{112} (1991), no. 3, 701--708.

\bibitem{a} S. Axler, Down with determinants!, {\it Amer. Math. Monthly}, \textbf{102} (1995), no. 2, 139-154.

\bibitem{bah} T. Berm\'{u}dez, A. Bonilla, H. Zaway, $C_0$-semigroups of $m$-isometries on Hilbert spaces, {\it J. Math. Anal. Appl.}, {\bf 472} (2019), no. 2, 879-893.

\bibitem{bdm} T. Berm\'{u}dez, C. D\'{\i}az Mendoza, A. Martin\'{o}n, Powers of $m$-isometries, {\it Studia Math.},  {\bf 208 } (2012), no. 3, 249-255.


\bibitem{bmne} T. Berm\'{u}dez, A. Martin\'{o}n, E. Negr\'{\i}n,  Weighted shift operators which are $m$-isometries, {\it Integral Equations Operator Theory}, {\bf  68} (2010), no. 3, 301-312.

\bibitem{bmn} T. Berm\'udez, A. Martin\'on, J. Noda, An isometry plus a nilpotent operator is an $m$-isometry. Applications, {\it J. Math. Anal. Appl.,} \textbf{407} (2013), no. 2, 505-512.

\bibitem{bmno} T. Berm\'udez, A. Martin\'on, J. Noda, Weighted shift and composition operators on $\ell_p$  which are $(m,q)$-isometries, {\it Linear Algebra Appl.},  {\bf 505} (2016), 152-173.

\bibitem{bmno 3} T. Berm\'udez, A. Martin\'on, V. M\"{u}ller, J. Noda, Perturbation of $m$-isometries by nilpotent operators, {\it Abstr. Appl. Anal.}, 2014, Art. ID 745479, 6 pp.


\bibitem{b} F. Botelho, On the existence of $n$-isometries on $\ell_p$ spaces, {\it Acta Sci. Math. (Szeged)}, \textbf{76} (2010), 183-192.

\bibitem{bjz} F. Botelho, J. E. Jamison, B. Zheng, Strict isometries of arbitrary orders, {\it Linear Algebra Appl.}, {\bf  436} (2012), no. 9, 3303-3314

\bibitem{m} M. Ch\={o}, S. Ôta, K.  Tanahashi, Invertible weighted shift operators which are $m$-isometries, {\it Proc. Amer. Math. Soc.}, {\bf  141} (2013), no. 12, 4241-4247.


\bibitem{d} B. P. Duggal,
Tensor product of $n$-isometries II. (English summary),
{\it Funct. Anal. Approx. Comput.}, {\bf  4 } (2012), no. 1, 27- 32.


\bibitem{F} W. H. Fleming, Functions of several variables. Second edition. Undergraduate Texts in Mathematics. Springer-Verlag, New York-Heidelberg, 1977.

\bibitem{eva} E. A. Gallardo-Guti\'{e}rrez, J. R. Partington,
$C_0$-semigroups of $2$-isometries and Dirichlet spaces,
{\it Revista Matematica Iberoamericana}, \textbf{34} (2018), no. 3, 1415-1425.

\bibitem{GR} I. S. Gradshteyn, I. M. Ryzhik, Table of integrals, Series, and Products, New York: Academic Press, (1980).

\bibitem{gc} C. Gu,   High order isometric composition operators on $\ell_p$ spaces and infinite graphs with polynomial growth, preprint 2019.

\bibitem{gu} C. Gu, Functional calculus for $m$-isometries and related operators on Hilbert spaces and Banach spaces, {\it Acta Sci. Math.},  (Szeged) {\bf 81} (2015), no. 3-4, 605-641.

\bibitem{gs} C. Gu, M.  Stankus,  Some results on higher order isometries and symmetries: Products and sums with a nilpotent operator, {\it  Linear Algebra Appl.}, {\bf  469} (2015), 500-509.

\bibitem{G} A.  Guivernau, Transformaciones que conservan el \'{a}rea, {\it Gaceta Matem\'{a}tica}, (1980), no. 5-6,  63-67.


\bibitem{j} Z. J. Jablonski, Complete hyperexpansivity, subnormality and inverted boundedness conditions,  {\it Integral Equations Operator Theory},  {\bf 44} (2002), no. 3, 316-336.
%[Theorem 2.3]

\bibitem{l} T. Le,  Algebraic properties of operator roots of polynomials, {\it J. Math. Anal. Appl.}, {\bf  421} (2015), no. 2, 1238-1246.


\bibitem{r} E. Rydhe, An Agler-type model theorem for $C_0$-semigroups of Hilbert space contractions,
 {\it J. Lond. Math. Soc.},  \textbf{93} (2016), no. 2, 420-438.


\bibitem{yh} S. Yarmahmoodi, K. Hedayatian, Isometric $N$-Jordan weighted shift operators, {\it Turk. J. Math.} DOI: 10.3906/mat-1507-54.

\bibitem{yhy} S. Yarmahmoodi, K.  Hedayatian, B.  Yousefi,
Supercyclicity and hypercyclicity of an isometry plus a nilpotent. (English summary),
{\it Abstr. Appl. Anal.}, (2011), Art. ID 686832, 11 pp.

    \end{thebibliography}
\end{document}